\documentclass[12pt]{article}
\usepackage{e-jc}
\usepackage{amsmath}
\usepackage{amsthm}
\usepackage{amssymb}
\usepackage[xcolor=pst]{pstricks}
\usepackage{graphicx, pst-plot, pst-node, pst-text, pst-tree}
\usepackage{titlefoot}
\usepackage[small,it]{caption}
\usepackage[hang,flushmargin,symbol,multiple]{footmisc} 
\usepackage{mathdots} 
\usepackage{color}
\definecolor{lightgray}{rgb}{0.8, 0.8, 0.8}
\definecolor{darkgray}{rgb}{0.7, 0.7, 0.7}
\definecolor{darkblue}{rgb}{0, 0, .4}

\usepackage[hyperfootnotes=false,colorlinks=true,citecolor=black,linkcolor=black,urlcolor=blue]{hyperref}
\hypersetup{
        pdfpagemode=UseThumbs,
        pdftitle={Well-quasi-order for permutation graphs omitting a path and a clique},
        pdfsubject={Combinatorics},
        pdfauthor={Atminas, Brignall, Korpelainen, Lozin, and Vatter},
}

\newtheorem{theorem}{Theorem}[section]
\newtheorem{proposition}[theorem]{Proposition}
\newtheorem{lemma}[theorem]{Lemma}

\DeclareFontFamily{U}{MnSymbolF}{}
\DeclareSymbolFont{mnsymbols}{U}{MnSymbolF}{m}{n}
\DeclareFontShape{U}{MnSymbolF}{m}{n}{
<-6> MnSymbolF5
<6-7> MnSymbolF6
<7-8> MnSymbolF7
<8-9> MnSymbolF8
<9-10> MnSymbolF9
<10-12> MnSymbolF10
<12-> MnSymbolF12}{}
\DeclareMathSymbol{\bigominus}{\mathop}{mnsymbols}{55}
\let\bigoplus\undefined
\DeclareMathSymbol{\bigoplus}{\mathop}{mnsymbols}{63}

\newcommand{\Av}{\operatorname{Av}}
\newcommand{\C}{\mathcal{C}}
\newcommand{\D}{\mathcal{D}}
\newcommand{\HH}{\mathfrak{H}}

\newcommand{\M}{\mathcal{M}}
\newcommand{\U}{\mathcal{U}}
\newcommand{\zpm}{0/\mathord{\pm} 1}
\newcommand{\Grid}{\operatorname{Grid}}

\newcommand{\Perms}{\Pi} 
\newcommand{\st}{\::\:}
\newcommand{\fnmatrix}[2]{\text{$\left(\mbox{\begin{footnotesize}$\begin{array}{#1}#2\end{array}$\end{footnotesize}}\right)$}}
\newcommand{\sfu}{\mathsf{u}}
\newcommand{\sfd}{\mathsf{d}}
\newcommand{\sfr}{\mathsf{r}}
\newcommand{\sfl}{\mathsf{l}}
\newcommand{\rc}{\text{rc}}

\makeatletter
\newcommand\footnoteref[1]{\protected@xdef\@thefnmark{\ref{#1}}\@footnotemark}
\makeatother
\DefineFNsymbols*{numlamport}[math]{1234\dagger\ddagger\S\P\|{**}{\dagger\dagger}{\ddagger\ddagger}}
\setfnsymbol{numlamport}

\title{Well-Quasi-Order for Permutation Graphs Omitting a Path and a Clique}

\author{%
Aistis Atminas\footnoteref{fn-dimap}\\
\small DIMAP and Mathematics Institute\\[-0.8ex] 
\small University of Warwick, Coventry, UK\\
\small\tt a.atminas@warwick.ac.uk\\ 
\and
Robert Brignall\footnoteref{fn-epsrc}\\
\small Department of Mathematics and Statistics\\[-0.8ex]
\small The Open University, Milton Keynes, UK\\
\small\tt r.brignall@open.ac.uk\\ 
\and
Nicholas Korpelainen\footnoteref{fn-epsrc},\\
\small Mathematics Department\\[-0.8ex]
\small University of Derby, Derby, UK\\
\small\tt n.korpelainen@derby.ac.uk\\ 
\and
Vadim Lozin\footnoteref{fn-dimap}\footnoteref{fn-lozinepsrc}\\
\small DIMAP and Mathematics Institute\\[-0.8ex] 
\small University of Warwick, Coventry, UK\\
\small\tt v.lozin@warwick.ac.uk\\ 
\and 
Vincent Vatter\footnoteref{fn-epsrc}\footnoteref{fn-vatter}\\
\small Department of Mathematics\\[-0.8ex]
\small University of Florida, Gainesville, Florida, USA\\
\small\tt vatter@ufl.edu\\ 
}

\date{\dateline{Jan 31, 2014}{Apr 2, 2015}}

\begin{document}

\addtocounter{footnote}{1}
\footnotetext{\label{fn-dimap}%
Atminas and Lozin gratefully acknowledge support from DIMAP -- the Center for Discrete Mathematics and its Applications at the University of Warwick.%
}
\addtocounter{footnote}{1}
\footnotetext{\label{fn-epsrc}%
Brignall, Korpelainen, and Vatter were partially supported by EPSRC Grant EP/J006130/1.%
}
\addtocounter{footnote}{1}
\footnotetext{\label{fn-lozinepsrc}%
Lozin was partially supported by EPSRC Grants EP/I01795X/1 and EP/L020408/1.%
}

\addtocounter{footnote}{1}
\footnotetext{%
Vatter was partially supported by the National Security Agency under Grant Number H98230-12-1-0207 and the National Science Foundation under Grant Number DMS-1301692.  The United States Government is authorized to reproduce and distribute reprints not-withstanding any copyright notation herein.
\label{fn-vatter}%
}
\maketitle

\begin{abstract}
We consider well-quasi-order for classes of permutation graphs which omit both a path and a clique. Our principle result is that the class of permutation graphs omitting $P_5$ and a clique of any size is well-quasi-ordered. This is proved by giving a structural decomposition of the corresponding permutations. We also exhibit three infinite antichains to show that the classes of permutation graphs omitting $\{P_6,K_6\}$, $\{P_7,K_5\}$, and $\{P_8,K_4\}$ are not well-quasi-ordered.
\end{abstract}

\section{Introduction}\label{sec-intro}

While the Minor Theorem of Robertson and Seymour~\cite{robertson:graph-minors-i-xx:} shows that the set of all graphs is well-quasi-ordered under the minor relation, it is well known that this set is not well-quasi-ordered under the induced subgraph order. Consequently, there has been considerable interest in determining which classes of graphs are well-quasi-ordered under this order. Here we consider finite graphs and permutation graphs which omit both a path $P_k$ and a clique $K_\ell$. Our main result, proved in Section~\ref{sec-p5-kell}, is that permutation graphs which avoid both $P_5$ and a clique $K_\ell$ are well-quasi-ordered under the induced subgraph order for every finite $\ell$. We also prove, in Section~\ref{sec-p7}, that the three classes of permutation graphs defined by forbidding $\{P_6,K_6\}$, $\{P_7,K_5\}$ and $\{P_8,K_4\}$ respectively are not well-quasi-ordered, by exhibiting an infinite antichain in each case.

We begin with elementary definitions. We say that a \emph{class} of graphs is a set of graphs closed under isomorphism and taking induced subgraphs. A class of graphs is \emph{well-quasi-ordered (wqo)} if it contains neither an infinite strictly decreasing sequence nor an infinite antichain (a set of pairwise incomparable graphs) under the induced subgraph ordering. Note that as we are interested only in classes of finite graphs, wqo is synonymous with a lack of infinite antichains.

Given a permutation $\pi=\pi(1)\cdots\pi(n)$, its corresponding \emph{permutation graph} is the graph $G_\pi$ on the vertices $\{1,\dots,n\}$ in which $i$ is adjacent to $j$ if both $i\le j$ and $\pi(i)\ge\pi(j)$. This mapping is many-to-one, because, for example, $G_{231}\cong G_{312}\cong P_3$. Given permutations $\sigma=\sigma(1)\cdots\sigma(k)$ and $\pi=\pi(1)\cdots\pi(n)$, we say that $\sigma$ is \emph{contained} in $\pi$ and write $\sigma\le\pi$ if there are indices $1\le i_1<\cdots<i_k\le n$ such that the sequence $\pi(i_1)\cdots\pi(i_k)$ is in the same relative order as $\sigma$. For us, a \emph{class} of permutations is a set of permutation closed downward under this containment order, and a class is wqo if it does not contain an infinite antichain.

The mapping $\pi\mapsto G_\pi$ is easily seen to be \emph{order-preserving}, i.e., if $\sigma\le\pi$ then $G_\sigma$ is an induced subgraph of $G_\pi$. Therefore if a class $\C$ of permutations is wqo then the associated class of permutation graphs $\{G_\pi \st \pi\in\C\}$ must also be wqo. However, it is possible to exploit the many-to-one nature of this mapping to construct infinite antichains of permutations which do not correspond to antichains of permutation graphs (though currently there are no examples of this in the literature). For this reason, when showing that classes of permutation graphs are wqo we instead prove the stronger result that the associated permutation classes are wqo, but when constructing infinite antichains, we must construct antichains of permutation graphs.

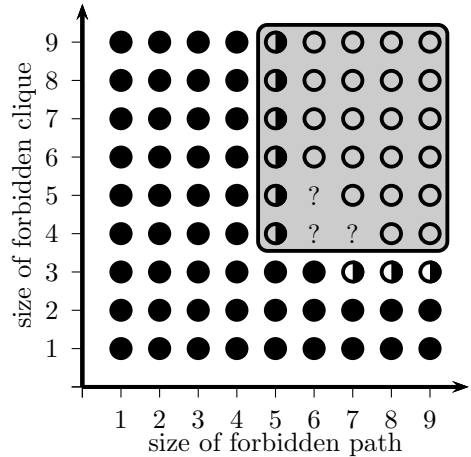
\begin{figure}[t]
\begin{center}
\begin{footnotesize}

\newcommand{\wqobothicon}{\pscircle*{0.3\psxunit}}
\newcommand{\wqopermsicon}{%
	\psarc*(0,0){0.3\psxunit}{-90}{90}
	\pscircle{0.3\psxunit}
}
\newcommand{\wqononeicon}{\pscircle{0.3\psxunit}}

\psset{xunit=0.2in, yunit=0.2in, linestyle=solid, linewidth=0.02in}
\begin{pspicture}(0,-2)(10,10)
\psaxes[dy=1,Dy=1,dx=1,Dx=1,tickstyle=bottom,showorigin=false]{->}(0,0)(10,10)
\rput[c](5,-1.5){\mbox{size of forbidden path}}
\rput[c](-1.5,5){\rotateleft{\mbox{size of forbidden clique}}}

\psframe[framearc=0.1,fillcolor=lightgray,fillstyle=solid](4.5,3.5)(9.5,9.5)

\multips(1,1)(1,0){4}{
	\multips(0,0)(0,1){9}{%
		\rput(0,0){\wqobothicon}
	}
}

\multips(5,1)(1,0){5}{%
	\rput(0,0){\wqobothicon}
}

\multips(5,2)(1,0){5}{%
	\rput(0,0){\wqobothicon}
}

\rput(5,3){\wqobothicon}
\rput(6,3){\wqobothicon}
\rput(7,3){\wqopermsicon}
\rput(8,3){\wqopermsicon}
\rput(9,3){\wqopermsicon}

\multips(5,4)(0,1){6}{
	\rput(0,0){\wqopermsicon}
}

\rput(6,4){?}
\rput(6,5){?}
\multips(6,6)(0,1){4}{
	\rput(0,0){\wqononeicon}
}

\rput(7,4){?}
\multips(7,5)(0,1){5}{
	\rput(0,0){\wqononeicon}
}

\multips(8,4)(0,1){6}{
	\rput(0,0){\wqononeicon}
}

\multips(9,4)(0,1){6}{
	\rput(0,0){\wqononeicon}
}

\end{pspicture}

\end{footnotesize}
\end{center}
\caption{This figure shows the known wqo results for classes of graphs and permutation graphs avoiding paths and cliques, including the results of this paper. Filled circles indicate that all graphs avoiding the specified path and clique are wqo. Half-filled circles indicate that the corresponding class of permutation graphs are wqo, but that the corresponding class of all graphs are not wqo. Empty circles indicate that neither class is wqo. Note that for the three unknown cases (indicated by question marks), it is known that the corresponding class of graphs contains an infinite antichain.}\label{fig-wqo-results}
\end{figure}

A summary of our results is shown in Figure~\ref{fig-wqo-results}, with our new contributions in the upper-right highlighted. The rest of this paper is organised as follows: in Section~\ref{sec-non-perm} we briefly summarise the status of the analogous question for non-permutation graphs. Section~\ref{sec-structural-tools} sets up the necessary notions from the study of permutation classes, before the proof of the well-quasi-orderability of $P_5$, $K_\ell$-free permutation graphs in Section~\ref{sec-p5-kell}. Section~\ref{sec-p7} contains three non-wqo results, Section~\ref{sec-enum} briefly presents some enumerative consequences of our results, and the final section contains a few concluding remarks about the three remaining open cases.

\section{Non-Permutation Graphs}\label{sec-non-perm}

When the graphs needn't be permutation graphs, well-quasi-ordering is of course harder to attain. On the side of wqo, graphs avoiding $K_2$ are trivially wqo and graphs avoiding $P_4$ (co-graphs) are well-known to be wqo, so there are only two nontrivial results, namely graphs avoiding $K_3$ and $P_5$ or $P_6$. Of course, it suffices to show that $K_3, P_6$-free graphs are wqo, and this was recently proved by Atminas and Lozin~\cite{atminas:labelled-induced:}.

On the side of non-wqo, two infinite antichains are required: one (from~\cite{korpelainen:two-forbidden-i:}) in the class of graphs with neither $P_5$ (or even $2K_2$) nor $K_4$, and one (from~\cite{korpelainen:bipartite-induc:}) in the class omitting both $P_7$ and $K_3$. For completeness, we outline both constructions here.

For graphs which omit $P_5$ and $K_4$, Korpelainen and Lozin~\cite{korpelainen:two-forbidden-i:} construct an infinite antichain by adapting a correspondence between permutations and graphs due to Ding~\cite{ding:subgraphs-and-w:}. The correspondence we require for our antichain can be described as follows, and is accompanied by Figure~\ref{fig-ding-antichain}. For a permutation $\pi$ of length $n$, first note that $\pi$ can be thought of as a structure with $n$ points, equipped with two linear orderings.

With this in mind, form a graph $B_\pi$ which consists of three independent sets, $U$, $V$ and $W$, each containing $n$ vertices. Let $U=\{u_1,u_2,\dots,u_n\}$. Between $U$ and $V$, there is a \emph{chain graph}: vertex $u_i$ in $U$ is adjacent to $i$ vertices of $V$, and for $i>1$ the neighbourhood of $u_{i-1}$ in $V$ is contained in the neighbourhood of $u_i$ in $V$. Note that this containment of neighbourhoods defines a linear ordering on the vertices of $U$: $u_1<u_2<\cdots < u_n$.

Next, between $U$ and $W$, we build another chain graph. This time, vertex $u_{\pi(i)}$ has $i$ neighbours in $W$, and for $i>1$ the neighbourhood of $u_{\pi(i)}$ in $W$ contains the neighbourhood of $u_{\pi(i-1)}$ in $W$. This defines a second linear ordering on $U$, namely $u_{\pi(1)}\prec u_{\pi(2)}\prec \cdots\prec u_{\pi(n)}$, and hence $\pi$ has been encoded in $B_\pi$. Finally, to complete the construction, between $V$ and $W$ there is a complete bipartite graph, i.e., every edge is present.

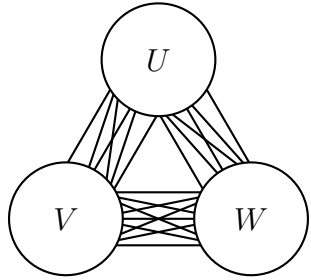
\begin{figure}
\centering
\psset{xunit=10pt, yunit=10pt, runit=0.03in}
\begin{pspicture}(-5,-5)(5,6)
\rput*{-30}(0,4){%
    \pnode(-1,-1){uw1}
    \pnode(0,-1){uw2}
    \pnode(1,-1){uw3}
}
\rput*{30}(0,4){%
    \pnode(-1,-1){uv1}
    \pnode(0,-1){uv2}
    \pnode(1,-1){uv3}
}
\rput*{-120}(0,0){%
 \rput*{-30}(0,4){%
    \pnode(-1,-1){vu3}
    \pnode(0,-1){vu2}
    \pnode(1,-1){vu1}
  }
  \rput*{30}(0,4){%
    \pnode(-1,-1){vw3}
    \pnode(0,-1){vw2}
    \pnode(1,-1){vw1}
  }
}
\rput*{120}(0,0){%
 \rput*{30}(0,4){%
    \pnode(-1,-1){wu3}
    \pnode(0,-1){wu2}
    \pnode(1,-1){wu1}
  }
  \rput*{-30}(0,4){%
    \pnode(-1,-1){wv1}
    \pnode(0,-1){wv2}
    \pnode(1,-1){wv3}
  }
}
\ncline{uv1}{vu1}\ncline{uv1}{vu2}\ncline{uv1}{vu3}
\ncline{uv2}{vu2}\ncline{uv2}{vu3}\ncline{uv3}{vu3}
\ncline{uw1}{wu1}\ncline{uw1}{wu2}\ncline{uw1}{wu3}
\ncline{uw2}{wu2}\ncline{uw2}{wu3}\ncline{uw3}{wu3}
\ncline{vw1}{wv1}\ncline{vw2}{wv2}\ncline{vw3}{wv3}
\ncline{vw1}{wv2}\ncline{vw2}{wv3}\ncline{vw3}{wv1}
\ncline{vw1}{wv3}\ncline{vw2}{wv1}\ncline{vw3}{wv2}
\pscircle[fillstyle=solid,fillcolor=white](0,4){10}\rput(0,4){$U$}
\rput*{120}(0,0){\pscircle[fillstyle=solid,fillcolor=white](0,4){10}\rput*{*0}(0,4){$V$}}
\rput*{-120}(0,0){\pscircle[fillstyle=solid,fillcolor=white](0,4){10}\rput*{*0}(0,4){$W$}}
\end{pspicture}
\caption[]{The construction of $B_\pi$ from a permutation $\pi$.}
\label{fig-ding-antichain}
\end{figure}

Now it is routine to verify that $B_\pi$ does not contain $2K_2\le P_5$ or $K_4$ for any $\pi$. Moreover for permutations $\sigma$ and $\pi$, we have $\sigma\leq\pi$ if and only if $B_\sigma\leq B_\pi$ as induced subgraphs. Thus, one may take any infinite antichain of permutations (for example, the ``increasing oscillating'' antichain), and encode each element of the antichain as a graph, yielding an infinite antichain in the class of $2K_2$, $K_4$-free graphs.

For graphs which omit $P_7$ and $K_3$, a modification of this construction was used by~\cite{korpelainen:bipartite-induc:} to answer a question of Ding~\cite{ding:subgraphs-and-w:}, who had asked whether $P_7$-free bipartite graphs are wqo. Starting with $B_\pi$, the modification ``splits'' every vertex $u$ of $U$ into two, $u^{(1)}$ and $u^{(2)}$, with an edge between them. The new vertex $u^{(1)}$ takes the neighbourhood of $u$ with $V$, while $u^{(2)}$ takes the neighbourhood of $u$ with $W$. The result is a graph with four independent sets $U^{(1)}$, $U^{(2)}$, $V$ and $W$ each of size $n$, with a perfect matching between $U^{(1)}$ and $U^{(2)}$, a complete bipartite graph between $V$ and $W$, and chain graphs between $U^{(1)}$ and $V$, and between $U^{(2)}$ and $W$.

This construction yields a $2P_3$-free bipartite graph. Moreover, the permutation $\pi$ is still encoded in such a way as to ensure an infinite antichain of permutations maps to an infinite antichain of graphs.

\section{Structural Tools}\label{sec-structural-tools}

Instead of working directly with permutation graphs, we establish our wqo results for the corresponding permutation classes (which, by our observations in the introduction, is a stronger result). The permutations $24153$ and $31524$ are the only permutations which correspond to the permutation graph $P_5$, and thus to establish our main result we must determine the structure of the permutation class
$$
\Av(24153, 31524, \ell\cdots 21).
$$
Considering the wqo problem from the viewpoint of permutations has the added benefit of allowing us to make use of the recently developed tools in this field. In particular, we utilise grid classes and the substitution decomposition. Thus before establishing our main result we must first introduce these concepts.

We frequently identify a permutation $\pi$ of length $n$ with its \emph{plot}, the set $\{(i,\pi(i))\st 1\le i\le n\}$ of points in the plane. We say that a rectangle in the plane is \emph{axis-parallel} if its top and bottom sides are parallel with the $x$-axis while its left and right sides are parallel with the $y$-axis. Given natural numbers $i$ and $j$ we denote by $[i,j]$ the closed interval $\{i,i+1,\dots,j\}$ and by $[i,j)$ the half-open interval $\{i,i+1,\dots,j-1\}$. Thus the axis-parallel rectangles we are interested in may be described by $[x,x']\times[y,y']$ for natural numbers $x$, $x'$, $y$, and $y'$.

Monotone grid classes are a way of partitioning the entries of a permutation (or rather, its plot) into monotone axis-parallel rectangles in a manner specified by a $\zpm$ matrix. In order for these matrices to align with plots of permutations, we index them with Cartesian coordinates. Suppose that $M$ is a $t\times u$ matrix (thus $M$ has $t$ columns and $u$ rows). An \emph{$M$-gridding} of the permutation $\pi$ of length $n$ consists of a pair of sequences $1=c_1\le\cdots\le c_{t+1}=n+1$ and $1=r_1\le\cdots\le r_{u+1}=n+1$ such that for all $k$ and $\ell$, the entries of (the plot of) $\pi$ that lie in the axis-parallel rectangle $[c_k,c_{k+1})\times [r_\ell,r_{\ell+1})$ are increasing if $M_{k,\ell}=1$, decreasing if $M_{k,\ell}=-1$, or empty if $M_{k,\ell}=0$.

We say that the permutation $\pi$ is \emph{$M$-griddable} if it possesses an $M$-gridding, and the \emph{grid class} of $M$, denoted by $\Grid(M)$, consists of the set of $M$-griddable permutations. We further say that the permutation class $\C$ is $M$-griddable if $\C\subseteq\Grid(M)$, and that this class is \emph{monotone griddable} if there is a finite matrix $M$ for which it is $M$-griddable.

Grid classes were first described in this generality (albeit under a different name) by Murphy and Vatter~\cite{murphy:profile-classes:}, who studied their wqo properties. To describe their result we need the notion of the \emph{cell graph} of a matrix $M$. This graph has vertex set $\{(i,j)\st M_{i,j}\neq 0\}$ and $(i,j)$ is adjacent to $(k,\ell)$ if they lie in the same row or column and there are no nonzero entries lying between them in this row or column. We typically attribute properties of the cell graph of $M$ to $M$ itself; thus we say that $M$ is a forest if its cell graph is a forest.

\begin{theorem}[Murphy and Vatter~\cite{murphy:profile-classes:}]\label{thm-murphy-vatter}
The grid class $\Grid(M)$ is wqo if and only if $M$ is a forest.
\end{theorem}

(This is a slightly different form of the result than is stated in \cite{murphy:profile-classes:}, but the two forms are equivalent as shown by Vatter and Waton~\cite{vatter:on-partial-well:}, who also gave a much simpler proof of Theorem~\ref{thm-murphy-vatter}.)

The monotone griddable classes were characterised by Huczynska and Vatter~\cite{huczynska:grid-classes-an:}. In order to present this result we also need some notation. Given permutations $\sigma$ and $\tau$ of respective lengths $m$ and $n$, their \emph{sum} is the permutation $\pi\oplus\tau$ whose plot consists of the plot of $\tau$ above and to the right of the plot of $\sigma$. More formally, this permutation is defined by
$$
\sigma\oplus\tau(i)
=
\left\{\begin{array}{ll}
\sigma(i)&\mbox{for $1\le i\le m$,}\\
\tau(i-m)+m&\mbox{for $m+1\le i\le m+n$.}
\end{array}\right.
$$
The obvious symmetry of this operation (in which the plot of $\tau$ lies below and to the right of the plot of $\sigma$) is called the \emph{skew sum} of $\sigma$ and $\tau$ and is denoted $\sigma\ominus\tau$. We can now state the characterisation of monotone griddable permutation classes.

\begin{theorem}[Huczynska and Vatter~\cite{huczynska:grid-classes-an:}]
\label{thm-grid-characterisation}
The permutation $\C$ is monotone griddable if and only if it does not contain arbitrarily long sums of $21$ or skew sums of $12$.
\end{theorem}

The reader might note that the classes we are interested in are not monotone griddable, let alone $M$-griddable for a forest $M$. However, our proof will show that these classes can be built from a monotone griddable class via the ``substitution decomposition'', which we define shortly. Before this, though, we must introduce a few more concepts concerning monotone grid classes, the first two of which are alternative characterisations of monotone griddable classes.

Given a permutation $\pi$, we say that the axis-parallel rectangle $R$ is \emph{monotone} if the entries of $\pi$ which lie in $R$ are monotone increasing or decreasing (otherwise $R$ is \emph{non-monotone}). We say that the permutation $\pi$ can be \emph{covered by $s$ monotone rectangles} if there is a collection $\mathfrak{R}$ of $s$ monotone axis-parallel rectangles such that every point in the plot of $\pi$ lies in at least one rectangle in $\mathfrak{R}$. Clearly if $\C\subseteq\Grid(M)$ for a $t\times u$ matrix $M$ then every permutation in $\C$ can be covered by $tu$ monotone rectangles. To see the other direction, note that every permutation which can be covered by $s$ monotone rectangles is $M$-griddable for some matrix $M$ of size at most $(2s-1)\times (2s-1)$. There are only finitely many such matrices, say $M^{(1)}$, $\dots$, $M^{(m)}$, so their \emph{direct sum},
$$
\left(\begin{array}{ccc}
&&M^{(m)}\\
&\iddots&\\
M^{(1)}&&
\end{array}\right)
$$
is a finite matrix whose grid class contains all such permutations.%
\footnote{\label{fn-matrix-sum}%
By adapting this argument it follows that if every permutation in the class $\C$ lies in the grid class of a forest of size at most $t\times u$, then $\C$ itself lies in the grid class of a (possibly much larger) forest.}

This characterisation of monotone griddability is recorded in Proposition~\ref{prop-mono-rectangles} below, which also includes a third characterisation. We say that the line $L$ \emph{slices} the rectangle $R$ if $L\cap R\neq\emptyset$. If $\C\subseteq\Grid(M)$ for a $t\times u$ matrix then for every permutation $\pi\in\C$ there is a collection of $t+u$ horizontal and vertical lines (the grid lines) which slice every non-monotone axis-parallel rectangle of $\pi$. Conversely, every such collection of lines defines a gridding of $\pi$, completing the sketch of the proof of the following result.

\begin{proposition}[specialising Vatter~{\cite[Proposition 2.3]{vatter:small-permutati:}}]
\label{prop-mono-rectangles}
For a permutation class $\C$, the following are equivalent:
\begin{enumerate}
\item[(1)] $\C$ is monotone griddable,
\item[(2)] there is a constant $\ell$ such that for every permutation $\pi\in\C$ the set of non-monotone axis-parallel rectangles of $\pi$ can be sliced by a collection of $\ell$ horizontal and vertical lines, and
\item[(3)] there is a constant $s$ such that every permutation in $\C$ can be covered by $s$ monotone rectangles.
\end{enumerate}
\end{proposition}

We now move on to the substitution decomposition, which will allow us to build the classes we are interested in from grid classes of forests. An {\it interval\/} in the permutation $\pi$ is a set of contiguous indices $I=[a,b]$ such that the set of values $\pi(I)=\{\pi(i) : i\in I\}$ is also contiguous. Given a permutation $\sigma$ of length $m$ and nonempty permutations $\alpha_1,\dots,\alpha_m$, the {\it inflation\/} of $\sigma$ by $\alpha_1,\dots,\alpha_m$ --- denoted $\sigma[\alpha_1,\dots,\alpha_m]$ --- is the permutation of length $|\alpha_1|+\cdots+|\alpha_m|$ obtained by replacing each entry $\sigma(i)$ by an interval that is order isomorphic to $\alpha_i$ in such a way that the intervals themselves are order isomorphic to $\sigma$. Thus the sum and skew sum operations are particular cases of inflations: $\sigma\oplus\tau=12[\sigma,\tau]$ and $\sigma\ominus\tau=21[\sigma,\tau]$. Given two classes $\C$ and $\U$, the \emph{inflation} of $\C$ by $\U$ is defined as
\[
\C[\U]=\{\sigma[\alpha_1,\dots,\alpha_m]\st\mbox{$\sigma\in\C_m$ and $\alpha_1,\dots,\alpha_m\in\U$}\}.
\]
The class $\C$ is said to be \emph{substitution closed} if $\C[\C]=\C$. The \emph{substitution closure}, $\langle\C\rangle$, of a class $\C$ is defined as the smallest substitution closed class containing $\C$. A standard argument shows that $\langle\C\rangle$ exists, and by specialising a result of \cite{albert:inflations-of-g:} we obtain the following.

\begin{theorem}[Albert, Ru\v{s}kuc, and Vatter~{\cite[specialisation of Theorem 4.4]{albert:inflations-of-g:}}]
\label{thm-geom-simples-pwo-basis}
If the matrix $M$ is a forest then the class $\langle\Grid(M)\rangle$ is wqo.
\end{theorem}

The permutation $\pi$ is said to be \emph{simple} if the only ways to write it as an inflation are trivial --- that is, it can only be written either as the inflation of a permutation by singletons, or as the inflation of a singleton by a permutation. Thus if $\pi$ has length $n$, it is simple if and only if its only intervals have length $0$, $1$, and $n$. Thus every permutation can be expressed as the inflation of a simple permutation. Moreover, in most cases, this decomposition is unique. A permutation is said to be \emph{sum (resp., skew) decomposable} if it can be expressed as the sum (resp., skew sum) of two shorter permutations. Otherwise it is said to be \emph{sum (resp., skew) indecomposable}.

\begin{proposition}[Albert and Atkinson~\cite{albert:simple-permutat:}]\label{simple-decomp-1}
Every permutation $\pi$ except $1$ is the inflation of a unique simple permutation $\sigma$.  Moreover, if $\pi=\sigma[\alpha_1,\dots,\alpha_m]$ for a simple permutation $\sigma$ of length $m\ge 4$, then each interval $\alpha_i$ is unique.  If $\pi$ is an inflation of $12$ (i.e., is sum decomposable), then there is a unique sum indecomposable $\alpha_1$  such that $\pi=\alpha_1\oplus\alpha_2$.  The same holds, mutatis mutandis, with $12$ replaced by $21$ and sum replaced by skew.
\end{proposition}

We close this section by noting how easily this machinery can show that permutation graphs omitting both $P_k$ and $K_3$ are wqo for all $k$, a result originally due to Korpelainen and Lozin~\cite{korpelainen:bipartite-induc:}. The corresponding permutations all avoid $321$ and thus lie in the grid class of an \emph{infinite} matrix, known as the infinite staircase (see Albert and Vatter~\cite{albert:generating-and-:}). Moreover, the sum indecomposable permutations which avoid $321$ and the two permutations corresponding to $P_k$ can be shown to lie in the grid class of a finite staircase,
$$
\left(\begin{array}{ccccc}
&&&1&1\\
&&\iddots&\iddots\\
&1&1\\
1&1
\end{array}\right).
$$
Finally, it follows by an easy application of Higman's Theorem~\cite{higman:ordering-by-div:} that if the sum indecomposable permutations in a class are wqo, then the class itself is wqo. (In this case, the wqo conclusion also follows by Theorem~\ref{thm-geom-simples-pwo-basis}.)

\section{Permutation Graphs Omitting \texorpdfstring{$P_5$ and $K_\ell$}{P\_5 and K\_l}}\label{sec-p5-kell}

In this section we prove that the class of permutations corresponding to permutation graphs omitting $P_5$ and $K_\ell$,
$$
\Av(24153, 31524, \ell\cdots 21),
$$
is wqo. Our proof basically consists of two steps. First, we show that the simple permutations in these classes are monotone griddable, and then we show that these griddings can be refined to forests. The conclusion then follows from Theorem~\ref{thm-geom-simples-pwo-basis}.

Given a set of points in the plane, their \emph{rectangular hull} is defined to be the smallest axis-parallel rectangle containing all of them. We begin with a very simple observation about these simple permutations.

\begin{proposition}
\label{prop-simples-n1}
For every simple permutation $\pi\in\Av(24153, 31524)$, either its greatest entry lies to the left of its least entry, or its leftmost entry lies above its rightmost entry.
\end{proposition}
\begin{proof}
Suppose, for a contradiction, that $\pi$ is a simple permutation in $\Av(24153, 31524)$ such that its greatest entry lies to the right of its least, and its leftmost entry lies below its rightmost entry. Thus, these four extremal entries form the pattern $2143$, and the situation is as depicted in Figure~\ref{fig-extremal-points}(i). Since $\pi$ is simple, regions $A$ and $B$ cannot both be empty, so, without loss of generality, suppose that $A$ is non-empty and label the greatest entry in this region as point $x$.

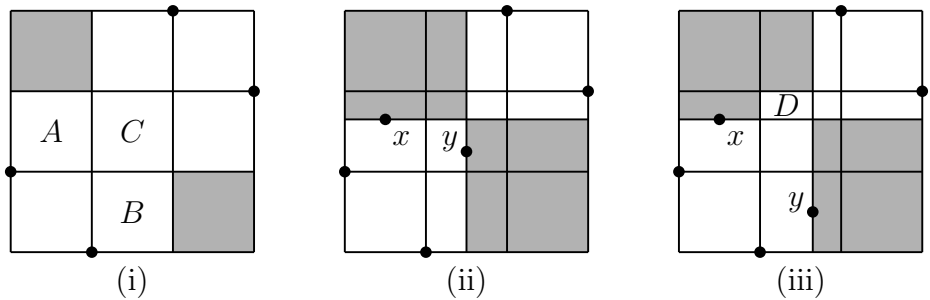
\begin{figure}
\centering
\begin{tabular}{ccccc}
\psset{xunit=0.21in, yunit=0.21in, runit=0.03in, linestyle=solid, linewidth=0.01in}
\begin{pspicture}(0,0)(6,6)
\pspolygon*[linecolor=darkgray](0,4)(0,6)(2,6)(2,4)
\pspolygon*[linecolor=darkgray](4,0)(6,0)(6,2)(4,2)
\rput[c](1,3){$A$}
\rput[c](3,1){$B$}
\rput[c](3,3){$C$}
\multips(0,0)(2,0){4}{\psline(0,0)(0,6)}
\multips(0,0)(0,2){4}{\psline(0,0)(6,0)}
\pscircle*(2,0){1}
\pscircle*(0,2){1}
\pscircle*(4,6){1}
\pscircle*(6,4){1}
\end{pspicture}
&\rule{10pt}{0pt}&
\psset{xunit=0.21in, yunit=0.21in, runit=0.03in, linestyle=solid, linewidth=0.01in}
\begin{pspicture}(0,0)(6,6)
\pspolygon*[linecolor=darkgray](0,3.3)(0,6)(3,6)(3,3.3)
\pspolygon*[linecolor=darkgray](3,0)(6,0)(6,3.3)(3,3.3)
\multips(0,0)(2,0){4}{\psline(0,0)(0,6)}
\multips(0,0)(0,2){4}{\psline(0,0)(6,0)}
\psline(0,3.3)(6,3.3)
\psline(3,0)(3,6)
\pscircle*(2,0){1}
\pscircle*(0,2){1}
\pscircle*(4,6){1}
\pscircle*(6,4){1}
\pscircle*(1,3.3){1}
\pscircle*(3,2.5){1}
\rput[t](1.4,3){$x$}
\rput[t](2.6,3){$y$}
\end{pspicture}
&\rule{10pt}{0pt}&
\psset{xunit=0.21in, yunit=0.21in, runit=0.03in, linestyle=solid, linewidth=0.01in}
\begin{pspicture}(0,0)(6,6)
\pspolygon*[linecolor=darkgray](0,3.3)(0,6)(3.3,6)(3.3,4)(2,4)(2,3.3)
\pspolygon*[linecolor=darkgray](3.3,0)(6,0)(6,3.3)(3.3,3.3)
\multips(0,0)(2,0){4}{\psline(0,0)(0,6)}
\multips(0,0)(0,2){4}{\psline(0,0)(6,0)}
\psline(0,3.3)(6,3.3)
\psline(3.3,0)(3.3,6)
\pscircle*(2,0){1}
\pscircle*(0,2){1}
\pscircle*(4,6){1}
\pscircle*(6,4){1}
\pscircle*(1,3.3){1}
\pscircle*(3.3,1){1}
\rput[t](1.4,3){$x$}
\rput[t](2.9,1.5){$y$}
\rput[c](2.65,3.65){$D$}
\end{pspicture}
\\
(i)&&(ii)&&(iii)\\[10pt]
\end{tabular}

\caption{The impossible configuration for a simple permutation in Proposition~\ref{prop-simples-n1}.}
\label{fig-extremal-points}
\end{figure}

Since $\pi$ is simple, the rectangular hull of the leftmost entry, the least entry, and the point $x$ cannot be an interval in $\pi$. Therefore, there must be a point either in $B$, or in that part of $C$ lying below $x$. Take the rightmost such point, and label it $y$. If $y$ is in region $C$, we immediately encounter the contradiction illustrated in Figure~\ref{fig-extremal-points}(ii): our choices of $x$ and $y$, and the forbidden permutation 24153 causes the permutation to be sum decomposable. Therefore, $y$ is placed in region $B$, and we have the picture depicted in Figure~\ref{fig-extremal-points}(iii). Since $\pi$ is simple, there must now be a point in the region labelled $D$. However, in order for this permutation to not be sum decomposable, we must insert another point in the region to the right of region $D$, and below the greatest entry in $D$, but this would force an occurrence of $31524$.\end{proof}

The class $\Av(24153, 31524, \ell\cdots 21)$ is closed under group-theoretic inversion (because $24153^{-1}=31524$ and $\ell\cdots 21^{-1}=\ell\cdots 21$), so we may always assume that the latter option in Proposition~\ref{prop-simples-n1} holds.

The rest of our proof adapts several ideas from Vatter~\cite{vatter:small-permutati:}. Two rectangles in the plane are said to be \emph{dependent} if their projections onto either the $x$- or $y$-axis have nontrivial intersection, and otherwise they are said to be \emph{independent}. A set of rectangles is called independent if its members are pairwise independent. Thus an independent set of rectangles may be viewed as a permutation, and it satisfies the Erd\H{o}s-Szekeres Theorem (every permutation of length at least $(a-1)(b-1)+1$ contains either $12\cdots a$ or $b\cdots 21$). We construct independent sets of rectangles in the proofs of both Propositions~\ref{prop-simples-griddable} and \ref{prop-gridding-corners}. In these settings, the rectangles are used to capture ``bad'' areas in the plot of a permutation, and our desired result is obtained by slicing the rectangles with horizontal and vertical lines in the sense of Proposition~\ref{prop-mono-rectangles}. The following result shows that we may slice a collection of rectangles with only a few lines, so long as we can bound its independence number.

\begin{theorem}[Gy{\'a}rf{\'a}s and Lehel~\cite{gyarfas:a-helly-type-pr:}]
\label{thm-slice}
There is a function $f(m)$ such that for any collection $\mathfrak{R}$ of axis-parallel rectangles in the plane which has no independent set of size greater than $m$, there exists a set of $f(m)$ horizontal and vertical lines which slice every rectangle in $\mathfrak{R}$.
\end{theorem}

Our next two results both rest upon Theorem~\ref{thm-slice}.

\begin{proposition}
\label{prop-simples-griddable}
For every $\ell$, the simple permutations in $\Av(24153, 31524, \ell\cdots 21)$ are contained in $\Grid(M)$ for a finite $0/1$ matrix $M$.
\end{proposition}
\begin{proof}
By Proposition~\ref{prop-mono-rectangles} it suffices to show that there is a function $g(\ell)$ such that every permutation in $\Av(24153, 31524, \ell\cdots 21)$ can be covered by $g(\ell)$ increasing axis-parallel rectangles (i.e., rectangles which only cover increasing sets of points). We prove this statement by induction on $\ell$. For the base case, we can take $g(2)=1$. Now take a simple permutation $\pi\in\Av(24153, 31524, \ell\cdots 21)$ for $\ell\ge 3$ and suppose that the claim holds for $\ell-1$.

\begin{figure}[t]
\begin{center}
\begin{tabular}{ccccc}

\psset{xunit=0.017in, yunit=0.017in, runit=0.03in}
\psset{linewidth=0.005in}
\begin{pspicture}(-17,0)(107,90)
\psline[linecolor=black]{c-c}(0,30)(90,30)
\psline[linecolor=black]{c-c}(0,60)(90,60)
\psline[linecolor=darkgray,linestyle=solid,linewidth=0.02in](0,0)(90,0)(90,90)(0,90)(0,0)
\pscircle*(0,60){1}
\pscircle*(90,30){1}
\rput[tl](-10,60){\rotateright{$\underbrace{\rule{60\psxunit}{0in}}$}}
\rput[r](-10,30){$\pi_b$}
\rput[tr](100,90){\rotateright{$\overbrace{\rule{60\psxunit}{0in}}$}}
\rput[l](100,60){$\pi_t$}
\end{pspicture}
&\rule{0in}{0pt}&
\psset{xunit=0.017in, yunit=0.017in, runit=0.03in}
\psset{linewidth=0.005in}
\begin{pspicture}(0,0)(90,90)
\rput(25,40){
	\psline[linecolor=darkgray,linestyle=solid,linewidth=0.02in]{c-c}(0,0)(10,0)(10,10)(0,10)(0,0)
	\pscircle*(0,10){1}
	\pscircle*[linecolor=white](10,0){1}
	\pscircle(10,0){1}
	\psline[linecolor=black]{c-c}(5,-30)(5,5)
	\pscircle*(5,-30){1}
}
\rput(55,70){
	\psline[linecolor=darkgray,linestyle=solid,linewidth=0.02in]{c-c}(0,0)(10,0)(10,10)(0,10)(0,0)
	\pscircle*(0,10){1}
	\pscircle*(10,0){1}
	\psline[linecolor=black]{c-c}(5,-50)(5,5)
	\pscircle*(5,-50){1}
}
\psline[linecolor=black]{c-c}(0,30)(90,30)
\psline[linecolor=black]{c-c}(0,60)(90,60)
\psline[linecolor=darkgray,linestyle=solid,linewidth=0.02in](0,0)(90,0)(90,90)(0,90)(0,0)
\pscircle*[linecolor=white](0,60){1}
\pscircle(0,60){1}
\pscircle*[linecolor=white](90,30){1}
\pscircle(90,30){1}
\end{pspicture}
&\rule{0in}{0pt}&
\psset{xunit=0.017in, yunit=0.017in, runit=0.03in}
\psset{linewidth=0.005in}
\begin{pspicture}(0,0)(90,90)
\rput(25,10){
	\psline[linecolor=darkgray,linestyle=solid,linewidth=0.02in]{c-c}(0,0)(10,0)(10,10)(0,10)(0,0)
	\pscircle*(0,10){1}
	\pscircle*(10,0){1}
	\psline[linecolor=black]{c-c}(5,60)(5,5)
	\pscircle*(5,60){1}
}
\rput(55,40){
	\psline[linecolor=darkgray,linestyle=solid,linewidth=0.02in]{c-c}(0,0)(10,0)(10,10)(0,10)(0,0)
	\pscircle*[linecolor=white](0,10){1}
	\pscircle(0,10){1}
	\pscircle*(10,0){1}
	\psline[linecolor=black]{c-c}(5,40)(5,5)
	\pscircle*(5,40){1}
}
\psline[linecolor=black]{c-c}(0,30)(90,30)
\psline[linecolor=black]{c-c}(0,60)(90,60)
\psline[linecolor=darkgray,linestyle=solid,linewidth=0.02in](0,0)(90,0)(90,90)(0,90)(0,0)
\pscircle*[linecolor=white](0,60){1}
\pscircle(0,60){1}
\pscircle*[linecolor=white](90,30){1}
\pscircle(90,30){1}
\end{pspicture}

\end{tabular}
\end{center}
\caption{On the left, the division of $\pi$ used in Proposition~\ref{prop-simples-griddable}. In the centre and on the right, the final contradictions in the proof.}
\label{fig-simples-griddable}
\end{figure}
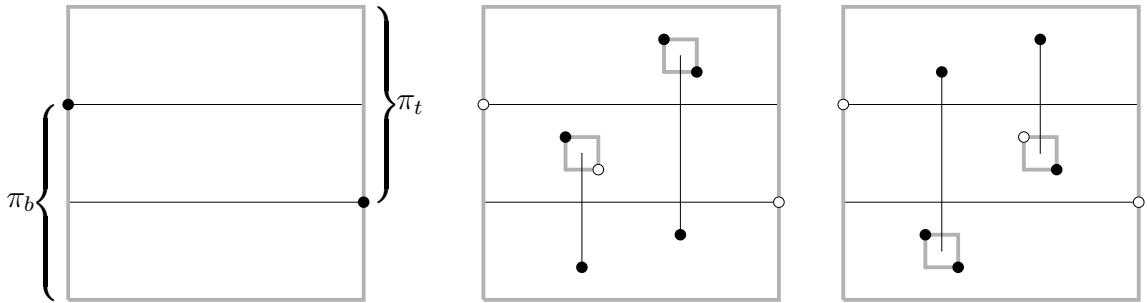

By Proposition~\ref{prop-simples-n1}, we may assume that the leftmost entry of $\pi$ lies above its rightmost entry. Let $\pi_t$ be the permutation formed by all entries of $\pi$ lying above its rightmost entry and $\pi_b$ the permutation formed by all entries of $\pi$ lying below its leftmost entry (as shown in Figure~\ref{fig-simples-griddable}). Thus every entry of $\pi$ corresponds to an entry in $\pi_t$, to an entry in $\pi_b$, or to entries in both permutations. Moreover, both $\pi_t$ and $\pi_b$ avoid $(\ell-1)\cdots 21$.

We would like to apply induction to find monotone rectangle coverings of both $\pi_t$ and $\pi_b$ but of course these permutations needn't be simple. Nevertheless, if $\pi_t$ is the inflation of the simple permutation $\sigma_t$ and $\pi_b$ of $\sigma_b$ then both $\sigma_t$ and $\sigma_b$ can be covered by $g(\ell-1)$ increasing axis-parallel rectangles by induction. Now we stretch these rectangles so that they cover the corresponding regions of $\pi$. By adapting the proof of Proposition~\ref{prop-mono-rectangles}, we may then extend this rectangle covering to a gridding of size at most $4g(\ell-1)\times 4g(\ell-1)$. While this gridding needn't be monotone, inside each of its cells we see points which correspond either to inflations of increasing sequences of $\sigma_t$ or of $\sigma_b$. Let $\mathfrak{L}$ denote the grid lines of this gridding.

We now say that the axis-parallel rectangle $R$ is \emph{bad} if it is fully contained in a cell of the above gridding and the points it covers contain a decreasing interval in either $\pi_t$ or $\pi_b$. Further let $\mathfrak{R}$ denote the collection of all bad rectangles. We aim to show that there is a collection of $f(2\ell(\ell-1))$ lines which slice every bad rectangle, where $f$ is the function defined in Theorem~\ref{thm-slice}. This, together with Proposition~\ref{prop-mono-rectangles} and the comments before it, will complete the proof because these lines together with $\mathfrak{L}$ will give a gridding of $\pi$ of bounded size.

Theorem~\ref{thm-slice} will give us the desired lines if we can show that $\mathfrak{R}$ has no independent set of size greater than $2\ell(\ell-1)+1$. Suppose to the contrary that $\mathfrak{R}$ does contain an independent set of this size. Thus at least $\ell(\ell-1)+1$ of these bad rectangles lie in one of $\pi_t$ or $\pi_b$; suppose first that these $\ell(\ell-1)+1$ bad rectangles lie in $\pi_t$. Because $\pi_t$ avoids $(\ell-1)\cdots 21$, the Erd\H{o}s-Szekeres Theorem implies that at least $\ell+1$ of its bad rectangles occur in increasing order (when read from left to right). Because $\pi$ itself is simple, each such rectangle must be separated, and this separating point must lie in $\pi_b\setminus\pi_t$ since (by definition) the points inside this bad rectangle form a decreasing interval in $\pi_t$. Appealing once more to Erd\H{o}s-Szekeres we see that two such separating points must themselves lie in increasing order, as shown in the centre of Figure~\ref{fig-simples-griddable}. However, this is a contradiction to our assumption that $\pi$ avoids $31524$ (given by the solid points). As shown on the rightmost pane of this figure, if the $\ell(\ell-1)+1$ bad rectangles lie in $\pi_b$ we instead find a copy of $24153$.
\end{proof}

A \emph{submatrix} of a matrix is obtained by deleting any collection of rows and columns from the matrix. Our next result shows that in $\Av(24153, 31524, \ell\cdots 21)$ the simple permutations can be gridded in a matrix which does not contain  $\fnmatrix{rr}{1&1\\1&*}$ as a submatrix, i.e., in this matrix, there is no non-zero cell with both a non-zero cell below it in the same column, and a non-zero cell to its right in the same row. (The $*$ indicates an entry that can be either 0 or 1.)

The following result is in some sense the technical underpinning of our entire argument. We advise the reader to note during the proof that if the hulls in $\HH$ are assumed to be increasing, then the resulting gridding matrix $M$ will be $0/1$, not $\zpm$.

\begin{proposition}
\label{prop-chop-hulls}
Suppose that $\pi$ is a permutation and $\HH$ is a collection of $m$ monotone rectangular hulls which cover the entries of $\pi$ satisfying
\begin{enumerate}
\item[(H1)] the hulls in $\HH$ are pairwise nonintersecting,
\item[(H2)] no single vertical or horizontal line slices through more than $k$ hulls in $\HH$, and
\item[(H3)] no hull in $\HH$ is dependent both with a hull to its right and a hull beneath it.
\end{enumerate}
Then there exists a function $f(m,k)$ such that $\pi$ is $M$-griddable for a $\zpm$ matrix $M$ of size at most $f(m,k)\times f(m,k)$ which does not contain $\fnmatrix{rr}{\pm1&\pm1\\\pm1&*}$ as a submatrix.
\end{proposition}

\begin{proof}
We define a gridding of $\pi$ using the sides of the hulls in $\HH$. For a given side from a given hull, form a gridline by extending it to the edges of the permutation. By our hypothesis (H2), this line can slice at most $k$ other hulls.

Whenever a hull is sliced in this way, a second gridline, perpendicular to the first, is induced so that all of the entries within the hull are contained in the bottom left and top right quadrants (for a hull containing increasing entries), or top left and bottom right quadrants (for a hull containing decreasing entries) defined by the two lines. This second gridline may itself slice through the interior of at most $k$ further hulls in $\HH$, and each such slice will induce another gridline, and so on. We call this process the \emph{propagation} of a line. See Figure~\ref{fig-propagation} for an illustration. For a given propagation sequence, the \emph{propagation tree} has gridlines for vertices, is rooted at the original gridline for the side, and has an edge between two gridlines if one induces the other in this propagation.
\begin{figure}
\centering
\begin{tabular}{ccc}
\psset{xunit=0.015in, yunit=0.015in}
\psset{linewidth=0.005in}
\begin{pspicture}(0,0)(150,115)
\rput(0,15){\psline[linecolor=darkgray,linestyle=solid,linewidth=0.01in]{c-c}(0,0)(150,0)}
\rput(0,55){\psline[linecolor=darkgray,linestyle=solid,linewidth=0.01in]{c-c}(0,0)(150,0)}
\rput(0,75){\psline[linecolor=darkgray,linestyle=solid,linewidth=0.01in]{c-c}(0,0)(150,0)}
\rput(0,90){\psline[linecolor=darkgray,linestyle=solid,linewidth=0.01in]{c-c}(0,0)(150,0)}
\rput(0,105){\psline[linecolor=darkgray,linestyle=solid,linewidth=0.01in]{c-c}(0,0)(150,0)}
\rput(10,0){\psline[linecolor=darkgray,linestyle=solid,linewidth=0.01in]{c-c}(0,0)(0,115)}
\rput(25,0){\psline[linecolor=darkgray,linestyle=solid,linewidth=0.01in]{c-c}(0,0)(0,115)}
\rput(40,0){\psline[linecolor=darkgray,linestyle=solid,linewidth=0.01in]{c-c}(0,0)(0,115)}
\rput(105,0){\psline[linecolor=darkgray,linestyle=solid,linewidth=0.01in]{c-c}(0,0)(0,115)}
\rput(135,0){\psline[linecolor=darkgray,linestyle=solid,linewidth=0.01in]{c-c}(0,0)(0,115)}
\rput(60,35){\psline[linecolor=black,linestyle=solid,linewidth=0.02in]{c-c}(0,0)(20,0)(20,20)(0,20)(0,0)}
\rput(5,10){\psline[linecolor=black,linestyle=solid,linewidth=0.02in]{c-c}(0,0)(10,0)(10,10)(0,10)(0,0)}
\rput(130,10){\psline[linecolor=black,linestyle=solid,linewidth=0.02in]{c-c}(0,0)(10,0)(10,10)(0,10)(0,0)}
\rput(20,50){\psline[linecolor=black,linestyle=solid,linewidth=0.02in]{c-c}(0,0)(10,0)(10,10)(0,10)(0,0)}
\rput(100,50){\psline[linecolor=black,linestyle=solid,linewidth=0.02in]{c-c}(0,0)(10,0)(10,10)(0,10)(0,0)}
\rput(130,50){\psline[linecolor=black,linestyle=solid,linewidth=0.02in]{c-c}(0,0)(10,0)(10,10)(0,10)(0,0)}
\rput(5,70){\psline[linecolor=black,linestyle=solid,linewidth=0.02in]{c-c}(0,0)(10,0)(10,10)(0,10)(0,0)}
\rput(20,85){\psline[linecolor=black,linestyle=solid,linewidth=0.02in]{c-c}(0,0)(10,0)(10,10)(0,10)(0,0)}
\rput(35,100){\psline[linecolor=black,linestyle=solid,linewidth=0.02in]{c-c}(0,0)(10,0)(10,10)(0,10)(0,0)}
\rput(130,100){\psline[linecolor=black,linestyle=solid,linewidth=0.02in]{c-c}(0,0)(10,0)(10,10)(0,10)(0,0)}
\end{pspicture}
&\rule{10pt}{0pt}&
\psset{xunit=0.015in, yunit=0.015in}
\psset{linewidth=0.005in}
\begin{pspicture}(0,-20)(70,90)
\cnode(30,80){0.04in}{a}
\cnode*(10,60){0.04in}{b}
\cnode*(10,40){0.04in}{c}
\cnode*(30,60){0.04in}{d}
\cnode*(50,60){0.04in}{e}
\cnode*(40,40){0.04in}{f}
\cnode*(40,20){0.04in}{g}
\cnode*(60,40){0.04in}{h}
\cnode*(60,20){0.04in}{i}
\cnode*(60,0){0.04in}{j}
\ncline{a}{b}
\ncline{a}{d}
\ncline{a}{e}
\ncline{b}{c}
\ncline{e}{f}
\ncline{e}{h}
\ncline{f}{g}
\ncline{h}{i}
\ncline{i}{j}
\end{pspicture}
\end{tabular}
\caption{Propagating a side of a hull in $\HH$, and its corresponding propagation tree.}
\label{fig-propagation}
\end{figure}
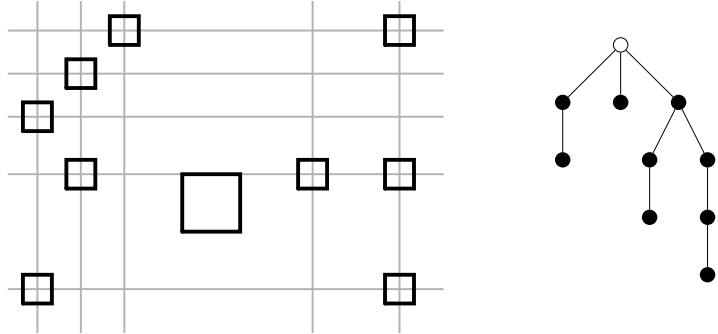

Before moving on, we note that it is clear that the propagation tree is connected, but less obvious that it is in fact a tree. This is not strictly required in our argument (we will only need to bound the number of vertices it contains), but if the tree were to contain a cycle it would have to correspond to a cyclic sequence of hulls, and this is impossible without contradicting hypothesis (H3).

In order to bound the size of a propagation tree, we first show that it has height at most $2m-1$. In the propagation tree of a side from some hull $H_0\in\HH$, take a sequence $H_1,H_2,\dots,H_p$ of hulls from $\HH$ corresponding to a longest path in the propagation tree, starting from the root. Thus, $H_1$ is sliced by the initial gridline formed from the side of $H_0$, and $H_i$ is sliced by the gridline induced from $H_{i-1}$ for $i=2,\dots,p$.

We now define a word $w=w_1w_2\cdots w_p$ from this sequence, based on the position of hull $H_i$ relative to hull $H_{i-1}$. For $i=1,2,\dots,p$, let
\[
w_i = \begin{cases}
\sfu & \text{if }H_i\text{ lies above }H_{i-1}\\
\sfd & \text{if }H_i\text{ lies below }H_{i-1}\\
\sfl & \text{if }H_i\text{ lies to the left of }H_{i-1}\\
\sfr & \text{if }H_i\text{ lies to the right of }H_{i-1}\\
\end{cases}
\]
Note that by the process of inducing perpendicular gridlines, successive letters in $w$ must alternate between $\{\sfu,\sfd\}$ and $\{\sfl,\sfr\}$. Moreover, since no hull interacts with hulls both below it and to its right, $w$ cannot contain a $\sfu\sfr$ or $\sfl\sfd$ factor. In other words, after the first instance of $\sfu$ or $\sfl$, there are no more instances of $\sfr$ or $\sfd$. This means that $w$ consists of a (possibly empty) alternating sequence $\sfd\sfr\sfd\sfr\cdots$ or $\sfr\sfd\sfr\sfd\cdots$ followed by a (possibly empty) alternating sequence $\sfu\sfl\sfu\sfl\cdots$ or $\sfl\sfu\sfl\sfu\cdots$.

Any alternating sequence of the form $\sfd\sfr\sfd\sfr\cdots$ or $\sfr\sfd\sfr\sfd\cdots$ can have at most $m-1$ letters, as each hull in $\HH$ (other then $H_0$) can be sliced at most once in such a sequence. Similarly, any alternating sequence of the form $\sfu\sfl\sfu\sfl\cdots$ or $\sfl\sfu\sfl\sfu\cdots$ can have at most $m-1$ letters. Consequently, we have $p\leq 2m-2$, and thus every propagation sequence has length at most $2m-1$, as required.

Since each gridline in the propagation tree has at most $k$ children, this means that the propagation tree for any given side has at most
\[1 + k + k^2 + \cdots + k^{2m-1} < k^{2m}\]
vertices, yielding a gridding of $\pi$ with fewer than $4mk^{2m}$ gridlines, and we may take this number to be $f(m,k)$. The gridding matrix $M$ is then naturally formed from the cells of this gridding of $\pi$: each empty cell corresponds to a $0$ in $M$, each cell containing points in decreasing order corresponds to a $-1$ in $M$, and each cell containing points in increasing order corresponds to a $1$ in $M$.

Finally, we verify that $M$ satisfies the conditions in the proposition. The process of propagating gridlines ensures that each rectangular hull in $\HH$ is divided into cells no two of which occupy the same row or column of the $M$-gridding. This means that there can be no  $\fnmatrix{rr}{\pm1&\pm1\\\pm1&*}$ submatrix of $M$ with two cells originating from the same hull in $\HH$. Thus, the cells of a submatrix of the form $\fnmatrix{rr}{\pm1&\pm1\\\pm1&*}$ must be made up from points in distinct hulls in $\HH$, but this is impossible since $\HH$ contains no hulls which are dependent with hulls both below it and to its right.
\end{proof}

We now apply this proposition to refine the gridding provided to us by Proposition~\ref{prop-simples-griddable}.

\begin{proposition}
\label{prop-gridding-corners}
For every $\ell$, the simple permutations in $\Av(24153, 31524, \ell\cdots 21)$ are contained in $\Grid(M)$ for a finite $0/1$ matrix $M$ which does not contain $\fnmatrix{rr}{1&1\\1&*}$ as a submatrix.
\end{proposition}

\begin{proof}
Let $\pi$ be an arbitrary simple permutation in $\Av(24153, 31524, \ell\cdots 21)$. As we observed in Footnote~\ref{fn-matrix-sum}, it suffices to show that there are constants $a$ and $b$ such that $\pi\in\Grid(M)$ for an $a\times b$ matrix $M$ satisfying the desired conditions.

By Proposition~\ref{prop-simples-griddable}, $\pi$ is contained in $\Grid(N)$ for some $0/1$ matrix $N$, of size (say) $t \times u$. We say that a \emph{bad} rectangle within any specified cell is an axis-parallel rectangle which contains two entries which are split both by points below and to the right. Since $\pi$ does not contain 24153, no cell can contain more than one independent bad rectangle --- see Figure~\ref{fig-bad-rectangles}. Therefore an independent set of bad rectangles can have size at most $tu$, so Theorem~\ref{thm-slice} shows that the bad rectangles can all be sliced by $f(tu)$ lines.

\begin{figure}
\centering
\psset{xunit=0.017in, yunit=0.017in, runit=0.03in}
\psset{linewidth=0.005in}
\begin{pspicture}(0,0)(90,90)
\psline[linecolor=darkgray,linestyle=solid,linewidth=0.02in](0,50)(40,50)(40,90)(0,90)(0,50)
\rput(5,55){
	\psline[linecolor=darkgray,linestyle=solid,linewidth=0.02in]{c-c}(0,0)(10,0)(10,10)(0,10)(0,0)
	\pscircle*(0,0){1}
	\pscircle[fillstyle=solid,fillcolor=white](10,10){1}
	\psline[linecolor=black]{c-c}(5,-40)(5,5)
	\pscircle[fillstyle=solid,fillcolor=white](5,-40){1}
	\psline[linecolor=black]{c-c}(80,5)(5,5)
	\pscircle*(80,5){1}
}%
\rput(25,75){
	\psline[linecolor=darkgray,linestyle=solid,linewidth=0.02in]{c-c}(0,0)(10,0)(10,10)(0,10)(0,0)
	\pscircle*(0,0){1}
	\pscircle*(10,10){1}
	\psline[linecolor=black]{c-c}(5,-50)(5,5)
	\pscircle*(5,-50){1}
	\psline[linecolor=black]{c-c}(50,5)(5,5)
	\pscircle[fillstyle=solid,fillcolor=white](50,5){1}
}%
\end{pspicture}
\caption{Two independent bad rectangles. The filled points form a copy of $24153$, irrespective of the relative orders of the splitting points.}\label{fig-bad-rectangles}
\end{figure}
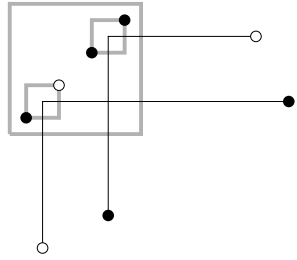

In any cell of the original gridding, the additional $f(tu)$ slices that have been added can at most slice these points into $f(tu)+1$ maximal unsliced pieces. In the entire permutation, therefore, these slices (together with the original gridlines from $N$) divide the points into at most $tu(f(tu)+1)$ maximal unsliced pieces. Let $\HH$ denote the rectangular hulls of these maximal unsliced pieces.

We now check that $\HH$ satisfies the hypotheses of Proposition~\ref{prop-chop-hulls}. Condition (H1) follows immediately by construction. Next, note that no two hulls of $\HH$ from the same cell of the $N$-gridding can be dependent, since every cell is monotone, so we may take $k=\max(t,u)$ to satisfy (H2). Finally, no hull can contain a bad rectangle (since all bad rectangles have been sliced), and so no hull can simultaneously be dependent with a hull from a cell below it, and a hull from a cell to its right, as required by (H3).

Now, applying Proposition~\ref{prop-chop-hulls} (noting that all the rectangular hulls in $\HH$ contain increasing entries), we have a $0/1$ gridding matrix $M_\pi$ for $\pi$, of dimensions at most $v\times w$ for some $v,w$, which does not contain $\fnmatrix{rr}{1&1\\1&*}$ as a submatrix. We are now done by our comments at the beginning of the proof.
\end{proof}

Having proved Proposition~\ref{prop-gridding-corners}, we merely need to put the pieces together to finish the proof of our main theorem. This proposition shows that there is a finite $0/1$ matrix $M$ with no submatrix of the form
$$
\fnmatrix{rr}{1&1\\1&*}
$$
such that the simple permutations of $\Av(24153, 31524, \ell\cdots 21)$ are contained in $\Grid(M)$. It follows that
$$
\Av(24153, 31524, \ell\cdots 21)\subseteq\langle\Grid(M)\rangle.
$$
Moreover, $M$ is a forest because if it were to contain a cycle, it would have to contain a submatrix of the form $\fnmatrix{rr}{1&1\\1&*}$. Therefore the permutation class $\Av(24153, 31524, \ell\cdots 21)$ is wqo by Theorem~\ref{thm-geom-simples-pwo-basis}.

\begin{theorem}
\label{thm-P5}
For every $\ell$, the permutation class $\Av(24153, 31524, \ell\cdots 21)$ is well-quasi-ordered. Therefore the class of permutation graphs omitting $P_5$ and $K_\ell$ is also well-quasi-ordered.
\end{theorem}


\section{Permutation Graphs Omitting \texorpdfstring{$P_6$, $P_7$ or $P_8$}{P\_6, P\_7 or P\_8}}\label{sec-p7}

In this section, we establish the following:

\begin{proposition}\label{prop-three-antichains}
The following three classes of graphs are not wqo:
\begin{enumerate}
\item[(1)] the $P_6$, $K_6$-free permutation graphs,
\item[(2)] the $P_7$, $K_5$-free permutation graphs, and
\item[(3)] the $P_8$, $K_4$-free permutation graphs.
\end{enumerate}
\end{proposition}

In order to prove these three classes are not wqo, it suffices to show that each class contains an infinite antichain. This is done by showing that the related permutation classes contain ``generalised'' grid classes, for which infinite antichains are already known. In general, we cannot immediately guarantee that the permutation antichain translates to a graph antichain, but we will show that this is in fact the case for the three we construct here.

We must first introduce \emph{generalised} grid classes. Suppose that $\M$ is a $t\times u$ matrix of permutation classes (we use calligraphic font for matrices containing permutation classes).  An {\it $\M$-gridding\/} of the permutation $\pi$ of length $n$ in this context is a pair of sequences $1=c_1\le\cdots\le c_{t+1}=n+1$ (the column divisions) and $1=r_1\le\cdots\le r_{u+1}=n+1$ (the row divisions) such that for all $1\le k\le t$ and $1\le\ell\le u$, the entries of $\pi$ from indices $c_k$ up to but not including $c_{k+1}$, which have values from $r_{\ell}$ up to but not including $r_{\ell+1}$ are either empty or order isomorphic to an element of $\M_{k,\ell}$.  The {\it grid class of $\M$\/}, written $\Grid(\M)$, consists of all permutations which possess an $\M$-gridding. The notion of monotone griddability can be analogously defined, but we do not require this.

Here, our generalised grid classes are formed from gridding matrices which contain the monotone class $\Av(21)$, and a non-monotone permutation class denoted $\bigoplus 21$. This is formed by taking all finite subpermutations of the infinite permutation $21436587\cdots$. In terms of minimal forbidden elements, we have $\bigoplus 21 = \Av(321,231,312)$.

Our proof of Proposition~\ref{prop-three-antichains} requires some further theory to ensure we can convert the permutation antichains we construct into graph antichains. The primary issue is that a permutation graph $G$ can have several different corresponding permutations. With this in mind, let \[\Perms(G) = \{\text{permutations } \pi : G_\pi \cong G\}\] denote the set of permutations each of which corresponds to the permutation graph $G$.

Denote by $\pi^{-1}$ the (group-theoretic) \emph{inverse} of $\pi$, by $\pi^\rc$ the \emph{reverse-complement} (formed by reversing the order of the entries of $\pi$, then replacing each entry $i$ by $|\pi|-i+1$), and by $(\pi^{-1})^{\rc}$ the \emph{inverse-reverse-complement}, formed by composing the two previous operations (in either order, as the two operations commute). It is then easy to see that if $\pi\in\Perms(G)$, then $\Perms(G)$ must also contain all of $\pi^{-1}$, $\pi^\rc$ and $(\pi^{-1})^{\rc}$. However, it is possible that $\Perms(G)$ may contain other permutations, and this depends on the graph-theoretic analogue of the substitution decomposition, which is called the \emph{modular decomposition}.

We also need to introduce the graph analogues of intervals and simplicity, which have different names in that context. A \emph{module} $M$ in a graph $G$ is a set of vertices such that for every $u,v\in M$ and $w\in V(G)\setminus M$, $u$ is adjacent to $w$ if and only if $v$ is adjacent to $w$. A graph $G$ is said to be \emph{prime} if it has no nontrivial modules, that is, any module $M$ of $G$ satisfies $|M| = 0, 1,$ or $|V(G)|$.

The following result, arising as a consequence of Gallai's work on transitive orientations, gives us some control over $\Perms(G)$:

\begin{proposition}[Gallai~\cite{gallai:transitiv-orien:}]\label{prop-gallai}
If $G$ is a prime permutation graph, then, up to the symmetries inverse, reverse-complement, and inverse-reverse-complement, $\Perms(G)$ contains a unique permutation.
\end{proposition}

We now extend Proposition~\ref{prop-gallai} to suit our purposes. Consider a simple permutation $\sigma$ of length $m\geq 4$, and form the permutation $\pi$ by inflating two\footnote{It is, in fact, possible to inflate more than two entries and establish the same result, but we do not require this here.} of the entries of $\sigma$ each by the permutation $21$. Note that $G_{21} = K_2$, and $\Perms(K_2)=\{21\}$. In the correspondence between graphs and permutations, modules map to intervals and vice versa, so prime graphs correspond to simple permutations, and there is an analogous result to Proposition~\ref{simple-decomp-1} for (permutation) graphs.

Thus, for any permutation $\rho\in\Perms(G_\pi)$, it follows that $\rho$ must be constructed by inflating two entries of some simple permutation $\tau$ by the permutation $21$. Moreover, $\tau$ must be one of $\sigma$, $\sigma^{-1}$, $\sigma^\rc$ or $(\sigma^{-1})^{\rc}$, and the entries of $\tau$ which are inflated are determined by which entries of $\sigma$ were inflated, and which of the four symmetries of $\sigma$ is equal to $\tau$. In other words, we still have $\Perms(G_\pi)=\{\pi,\pi^{-1},\pi^\rc,(\pi^{-1})^{\rc}\}$.

We require one further easy observation:

\begin{lemma}\label{lem-graph-perm-order}
If $G$ and $H$ are permutation graphs such that $H\leq G$, then for any $\pi\in\Perms(G)$ there exists $\sigma\in\Perms(H)$ such that $\sigma\leq\pi$.
\end{lemma}

\begin{proof}
Given any $\pi\in\Perms(G)$, let $\sigma$ denote the subpermutation of $\pi$ formed from the entries of $\pi$ which correspond to the vertices of an embedding of $H$ as an induced subgraph of $G$. Clearly $G_\sigma\cong H$, so $\sigma\in\Perms(H)$ and $\sigma\leq \pi$, as required.
\end{proof}

We are now in a position to prove Proposition~\ref{prop-three-antichains}. Since the techniques are broadly similar for all three cases, we will give the details for case (2), and only outline the key steps for the other two cases.

\begin{proof}[Proof of Proposition~\ref{prop-three-antichains}\,(2)]
First, the graph $P_7$ corresponds to two permutations, namely $3152746$ and $2416375$ respectively. Thus the class of $P_7$, $K_5$-free permutation graphs corresponds to the permutation class $\Av(3152746,2416375,54321)$. This permutation class contains the grid class \[\Grid\fnmatrix{rr}{ \bigoplus 21 & \bigoplus 21},\] because this grid class avoids the permutations $241635$ (contained in both $3152746$ and $2416375$), and $54321$.

We now follow the recipe given by Brignall~\cite{brignall:grid-classes-an:} to construct an infinite antichain which lies in $\Grid\fnmatrix{rr}{ \bigoplus 21 & \bigoplus 21}$. Call the resulting antichain $A$, the first three elements and general term of which are illustrated in Figure~\ref{fig-parallel-antichain}. This antichain is related to the ``parallel'' antichain in Murphy's thesis~\cite{murphy:restricted-perm:}.
\begin{figure}
\centering
\begin{tabular}{ccccccc}
\psset{xunit=0.4pt, yunit=0.4pt, runit=0.03in}
\begin{pspicture}(0,0)(130,130)
\psline[linecolor=gray](15,25)(80,10)(70,50)(40,40)(30,70)(100,60)(90,90)(60,80)(50,120)(115,105)
\pscircle*[linecolor=black](10,30){1.0}
\pscircle*[linecolor=black](20,20){1.0}
\pscircle*[linecolor=black](30,70){1.0}
\pscircle*[linecolor=black](40,40){1.0}
\pscircle*[linecolor=black](50,120){1.0}
\pscircle*[linecolor=black](60,80){1.0}
\psline[linecolor=gray](65,0)(65,130)
\pscircle*[linecolor=black](70,50){1.0}
\pscircle*[linecolor=black](80,10){1.0}
\pscircle*[linecolor=black](90,90){1.0}
\pscircle*[linecolor=black](100,60){1.0}
\pscircle*[linecolor=black](110,110){1.0}
\pscircle*[linecolor=black](120,100){1.0}
\psccurve(8,18)(4,36)(22,32)(26,14)
\rput(100,80){\psccurve(8,18)(4,36)(22,32)(26,14)}
\end{pspicture}
&\rule{5pt}{0pt}&
\psset{xunit=0.4pt, yunit=0.4pt, runit=0.03in}
\begin{pspicture}(0,0)(170,170)
\psline[linecolor=gray](15,25)(100,10)(90,50)(40,40)(30,70)(120,60)(110,90)(60,80)(50,110)(140,100)(130,130)(80,120)(70,160)(155,145)
\pscircle*[linecolor=black](10,30){1}
\pscircle*[linecolor=black](20,20){1}
\pscircle*[linecolor=black](30,70){1}
\pscircle*[linecolor=black](40,40){1}
\pscircle*[linecolor=black](50,110){1}
\pscircle*[linecolor=black](60,80){1}
\pscircle*[linecolor=black](70,160){1}
\pscircle*[linecolor=black](80,120){1}
\psline[linecolor=gray](85,0)(85,170)
\pscircle*[linecolor=black](90,50){1}
\pscircle*[linecolor=black](100,10){1}
\pscircle*[linecolor=black](110,90){1}
\pscircle*[linecolor=black](120,60){1}
\pscircle*[linecolor=black](130,130){1}
\pscircle*[linecolor=black](140,100){1}
\pscircle*[linecolor=black](150,150){1}
\pscircle*[linecolor=black](160,140){1}
\psccurve(8,18)(4,36)(22,32)(26,14)
\rput(140,120){\psccurve(8,18)(4,36)(22,32)(26,14)}
\end{pspicture}
&\rule{5pt}{0pt}&
\psset{xunit=0.4pt, yunit=0.4pt, runit=0.03in}
\begin{pspicture}(0,0)(210,210)
\psline[linecolor=gray](15,25)(120,10)(110,50)(40,40)(30,70)(140,60)(130,90)(60,80)(50,110)(160,100)(150,130)(80,120)(70,150)(180,140)(170,170)(100,160)(90,200)(195,185)
\pscircle*[linecolor=black](10,30){1}
\pscircle*[linecolor=black](20,20){1}
\pscircle*[linecolor=black](30,70){1}
\pscircle*[linecolor=black](40,40){1}
\pscircle*[linecolor=black](50,110){1}
\pscircle*[linecolor=black](60,80){1}
\pscircle*[linecolor=black](70,150){1}
\pscircle*[linecolor=black](80,120){1}
\pscircle*[linecolor=black](90,200){1}
\pscircle*[linecolor=black](100,160){1}
\psline[linecolor=gray](105,0)(105,210)
\pscircle*[linecolor=black](110,50){1}
\pscircle*[linecolor=black](120,10){1}
\pscircle*[linecolor=black](130,90){1}
\pscircle*[linecolor=black](140,60){1}
\pscircle*[linecolor=black](150,130){1}
\pscircle*[linecolor=black](160,100){1}
\pscircle*[linecolor=black](170,170){1}
\pscircle*[linecolor=black](180,140){1}
\pscircle*[linecolor=black](190,190){1.0}
\pscircle*[linecolor=black](200,180){1.0}
\psccurve(8,18)(4,36)(22,32)(26,14)
\rput(180,160){\psccurve(8,18)(4,36)(22,32)(26,14)}
\end{pspicture}
&\rule{5pt}{0pt}&
\psset{xunit=0.4pt, yunit=0.4pt, runit=0.03in}
\begin{pspicture}(0,0)(210,210)
\psline[linecolor=gray](15,25)(120,10)(110,50)(40,40)(30,70)(140,60)(130,90)
\psline[linecolor=gray,linestyle=dashed](130,90)(95,85)
\psline[linecolor=gray,linestyle=dashed](115,125)(80,120)
\psline[linecolor=gray](80,120)(70,150)(180,140)(170,170)(100,160)(90,200)(195,185)
\pscircle*[linecolor=black](10,30){1}
\pscircle*[linecolor=black](20,20){1}
\pscircle*[linecolor=black](30,70){1}
\pscircle*[linecolor=black](40,40){1}
\pscircle*[linecolor=black](70,150){1}
\pscircle*[linecolor=black](80,120){1}
\pscircle*[linecolor=black](90,200){1}
\pscircle*[linecolor=black](100,160){1}
\psline[linecolor=gray](105,0)(105,210)
\pscircle*[linecolor=black](110,50){1}
\pscircle*[linecolor=black](120,10){1}
\pscircle*[linecolor=black](130,90){1}
\pscircle*[linecolor=black](140,60){1}
\pscircle*[linecolor=black](170,170){1}
\pscircle*[linecolor=black](180,140){1}
\pscircle*[linecolor=black](190,190){1.0}
\pscircle*[linecolor=black](200,180){1.0}
\psccurve(8,18)(4,36)(22,32)(26,14)
\rput(180,160){\psccurve(8,18)(4,36)(22,32)(26,14)}
\end{pspicture}
\end{tabular}
\caption[]{The first three elements and the general structure of the infinite antichain $A$ in $\Grid\fnmatrix{rr}{ \bigoplus 21 & \bigoplus 21}$. The grey lines indicate the sequence used to place the entries.}
\label{fig-parallel-antichain}
\end{figure}
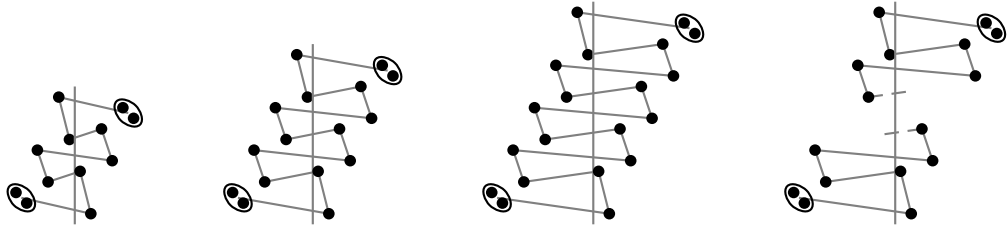

Now set $G_A = \{ G_\pi : \pi \in A\}$, and note that $G_A$ is contained in the class of permutation graphs omitting $P_7$ and $K_5$ by Lemma~\ref{lem-graph-perm-order}. If $G_A$ is an antichain of graphs, then we are done, so suppose for a contradiction that there exists $G,H\in G_A$ with $H\leq G$. Take the permutation $\pi\in A$ for which $G_\pi= G$. Applying Lemma~\ref{lem-graph-perm-order}, there exists $\sigma \in \Perms(H)$ such that $\sigma\leq \pi$. We cannot have $\sigma\in A$ since $A$ is an antichain, so $\sigma$ must be some other permutation with the same graph. Choose $\tau\in A\cap\Perms(H)$.

It is easy to check that the only non-trivial intervals in any permutation in $A$ are the two pairs of points in Figure~\ref{fig-parallel-antichain} which are circled. Thus, $\tau$ is formed by inflating two entries of a simple permutation each by a copy of $21$. By the comments after Proposition~\ref{prop-gallai}, we conclude that $\sigma$ is equal to one of $\tau^{-1}$, $\tau^\rc$ or $(\tau^{-1})^{\rc}$.

Note that every permutation in $A$ is closed under reverse-complement, so $\tau=\tau^\rc$ and thus $\sigma\neq\tau^\rc$. If $\sigma=\tau^{-1}$, by inspection every element of $A$ contains a copy of $24531$, which means that $\sigma$ contains $(24531)^{-1}=51423$. Since $\sigma\leq \pi$, it follows that $\pi$ must also contain a copy of $51423$. However this is impossible, because $51423$ is not in $\Grid\fnmatrix{rr}{ \bigoplus 21 & \bigoplus 21}$.

Thus, we must have $\sigma= (\tau^{-1})^{\rc}$, in which case $\sigma$ must contain a copy of $\left((24531)^\rc\right)^{-1}=34251$. However, this permutation is also not in $\Grid\fnmatrix{rr}{ \bigoplus 21 & \bigoplus 21}$ so cannot be contained in $\pi$. Thus $\sigma\not\in\pi$, and from this final contradiction we conclude that $H\not\leq G$, so $G_A$ is an infinite antichain of permutation graphs, as required.
\end{proof}
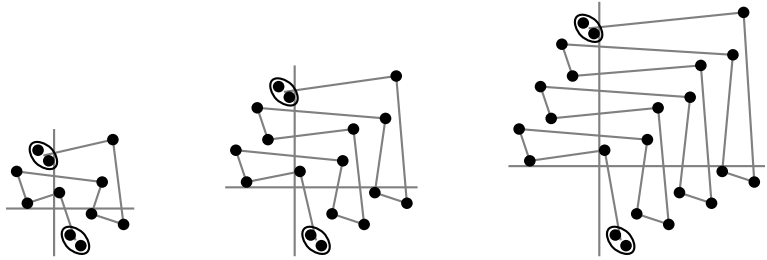
\begin{figure}
\centering
\begin{tabular}{ccccc}
\psset{xunit=0.4pt, yunit=0.4pt, runit=0.03in}
\begin{pspicture}(0,0)(120,120)
\psline[linecolor=gray](65,15)(50,60)(20,50)(10,80)(90,70)(80,40)(110,30)(100,110)(35,95)
\psline[linecolor=gray](45,0)(45,120)
\psline[linecolor=gray](0,45)(120,45)
\pscircle*[linecolor=black](10,80){1.0}
\pscircle*[linecolor=black](20,50){1.0}
\pscircle*[linecolor=black](30,100){1.0}
\pscircle*[linecolor=black](40,90){1.0}
\pscircle*[linecolor=black](50,60){1.0}
\pscircle*[linecolor=black](60,20){1.0}
\pscircle*[linecolor=black](70,10){1.0}
\pscircle*[linecolor=black](80,40){1.0}
\pscircle*[linecolor=black](90,70){1.0}
\pscircle*[linecolor=black](100,110){1.0}
\pscircle*[linecolor=black](110,30){1.0}
\rput(50,0){\psccurve(8,8)(4,26)(22,22)(26,4)}
\rput(20,80){\psccurve(8,8)(4,26)(22,22)(26,4)}
\end{pspicture}
&\rule{10pt}{0pt}&
\psset{xunit=0.4pt, yunit=0.4pt, runit=0.03in}
\begin{pspicture}(0,0)(180,180)
\psline[linecolor=gray](85,15)(70,80)(20,70)(10,100)(110,90)(100,40)(130,30)(120,120)(40,110)(30,140)(150,130)(140,60)(170,50)(160,170)(55,155)
\psline[linecolor=gray](65,0)(65,180)
\psline[linecolor=gray](0,65)(180,65)
\pscircle*[linecolor=black](10,100){1.0}
\pscircle*[linecolor=black](20,70){1.0}
\pscircle*[linecolor=black](30,140){1.0}
\pscircle*[linecolor=black](40,110){1.0}
\pscircle*[linecolor=black](50,160){1.0}
\pscircle*[linecolor=black](60,150){1.0}
\pscircle*[linecolor=black](70,80){1.0}
\pscircle*[linecolor=black](80,20){1.0}
\pscircle*[linecolor=black](90,10){1.0}
\pscircle*[linecolor=black](100,40){1.0}
\pscircle*[linecolor=black](110,90){1.0}
\pscircle*[linecolor=black](120,120){1.0}
\pscircle*[linecolor=black](130,30){1.0}
\pscircle*[linecolor=black](140,60){1.0}
\pscircle*[linecolor=black](150,130){1.0}
\pscircle*[linecolor=black](160,170){1.0}
\pscircle*[linecolor=black](170,50){1.0}
\rput(70,0){\psccurve(8,8)(4,26)(22,22)(26,4)}
\rput(40,140){\psccurve(8,8)(4,26)(22,22)(26,4)}
\end{pspicture}
&\rule{10pt}{0pt}&
\psset{xunit=0.4pt, yunit=0.4pt, runit=0.03in}
\begin{pspicture}(0,0)(240,240)
\psline[linecolor=gray](105,15)(90,100)(20,90)(10,120)(130,110)(120,40)(150,30)(140,140)(40,130)(30,160)(170,150)(160,60)(190,50)(180,180)(60,170)(50,200)(210,190)(200,80)(230,70)(220,230)(75,215)
\psline[linecolor=gray](85,0)(85,240)
\psline[linecolor=gray](0,85)(240,85)
\pscircle*[linecolor=black](10,120){1.0}
\pscircle*[linecolor=black](20,90){1.0}
\pscircle*[linecolor=black](30,160){1.0}
\pscircle*[linecolor=black](40,130){1.0}
\pscircle*[linecolor=black](50,200){1.0}
\pscircle*[linecolor=black](60,170){1.0}
\pscircle*[linecolor=black](70,220){1.0}
\pscircle*[linecolor=black](80,210){1.0}
\pscircle*[linecolor=black](90,100){1.0}
\pscircle*[linecolor=black](100,20){1.0}
\pscircle*[linecolor=black](110,10){1.0}
\pscircle*[linecolor=black](120,40){1.0}
\pscircle*[linecolor=black](130,110){1.0}
\pscircle*[linecolor=black](140,140){1.0}
\pscircle*[linecolor=black](150,30){1.0}
\pscircle*[linecolor=black](160,60){1.0}
\pscircle*[linecolor=black](170,150){1.0}
\pscircle*[linecolor=black](180,180){1.0}
\pscircle*[linecolor=black](190,50){1.0}
\pscircle*[linecolor=black](200,80){1.0}
\pscircle*[linecolor=black](210,190){1.0}
\pscircle*[linecolor=black](220,230){1.0}
\pscircle*[linecolor=black](230,70){1.0}
\rput(90,0){\psccurve(8,8)(4,26)(22,22)(26,4)}
\rput(60,200){\psccurve(8,8)(4,26)(22,22)(26,4)}
\end{pspicture}
\end{tabular}
\caption[]{The first three elements of an infinite antichain in $\Grid\fnmatrix{rr}{\bigoplus 21&\Av(21)\\&\bigoplus 21}$.}
\label{fig-hook-antichain}
\end{figure}
\begin{proof}[Sketch proof of Proposition~\ref{prop-three-antichains}\,(1) and (3)]
For (1), the class of $P_6$, $K_6$-free permutation graphs corresponds to $\Av(241635,315264,654321)$. This class contains \[\Grid\fnmatrix{rr}{\bigoplus 21&\Av(21)\\&\bigoplus 21}\] because this grid class does not contain $23154$ (contained in $241635$), $31254$ (contained in $315264$) and $654321$.

By~\cite{brignall:grid-classes-an:}, this grid class contains an infinite antichain whose first three elements are illustrated in Figure~\ref{fig-hook-antichain}. The permutations in this antichain have exactly two proper intervals, indicated by the circled pairs of points in each case, and this means that for any permutation $\pi$ in this antichain, $\Perms(G_\pi) = \{\pi, \pi^{-1}, \pi^\rc, (\pi^{-1})^{\rc}\}$.

Following the proof of Proposition~\ref{prop-three-antichains}\,(2), it suffices to show that for any permutations $\sigma$ and $\pi$ in the antichain, none of $\sigma$, $\sigma^{-1}$, $\sigma^\rc$, or $(\sigma^{-1})^{\rc}$ is contained in $\pi$. This is done by identifying permutations which are not contained in any antichain element $\pi$, but which are contained in one of the symmetries $\sigma^{-1}$, $\sigma^\rc$, or $(\sigma^{-1})^{\rc}$. We omit the details.

For (3), The $P_8$, $K_4$-free permutation graphs correspond to the permutation class $\Av(4321,24163857,31527486)$, and this contains the (monotone) grid class \[\Grid\fnmatrix{rr}{\Av(21)&\Av(21)\\\Av(21)&\Av(21)}.\] In this case, we appeal to Murphy and Vatter~\cite{murphy:profile-classes:} for an infinite antichain, and the same technique used to prove cases~(1) and~(2) can be applied.
\end{proof}

\section{Enumeration}\label{sec-enum}

One significant difference between studies of the induced subgraph order and the subpermutation order is the research interests of the different camps of investigators. In the latter area, a great proportion of the work is enumerative in nature. Therefore, having established the structure of permutation graphs avoiding $P_5$ and cliques, we briefly study the enumeration of the corresponding permutation classes in this section. Our goal is to show that these classes are \emph{strongly rational}, i.e., that they and every one of their subclasses have rational generating functions. Here we refer to
$$
\sum_{\pi\in\C} x^{|\pi|}
$$
as the generating function of the class $\C$, where $|\pi|$ denotes the length of the permutation $\pi$. Note that by a simple counting argument (made explicit in Albert, Atkinson, and Vatter~\cite{albert:subclasses-of-t:}), strongly rational permutation classes must be wqo.

Proposition~\ref{simple-decomp-1} allows us to associate with any permutation $\pi$ a unique \emph{substitution decomposition tree}. This tree is recursively defined by decomposing each node of the tree as
\begin{itemize}
\item $\sigma[\alpha_1,\dots,\alpha_m]$ where $\sigma$ is a nonmonotone simple permutation,
\item $\alpha_1\oplus\cdots\oplus\alpha_m$ where each $\alpha_i$ is sum indecomposable, or
\item $\alpha_1\ominus\cdots\ominus\alpha_m$ where each $\alpha_i$ is skew indecomposable.
\end{itemize}
See Figure~\ref{fig-tree-375896214} for an example. In particular, note that by our indecomposability assumptions, sum nodes (resp., skew sum nodes) cannot occur twice in a row when reading up a branch of the tree. The {\it substitution depth\/} of $\pi$ is then the height of its substitution decomposition tree, so for example, the substitution depth of the permutation from Figure~\ref{fig-tree-375896214} is $3$, while the substitution depth of any simple (or monotone) permutation is $1$. As we show in our next result, substitution depth is bounded for the classes we are interested in. This result is a special case of Vatter~\cite[Proposition 4.2]{vatter:small-permutati:}, but we include a short proof for completeness.

\begin{figure}
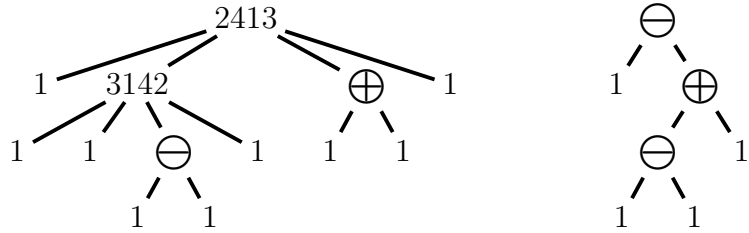

$$
\begin{array}{ccc}
\pstree[nodesep=3pt,levelsep=25pt,linewidth=0.02in]{\TR{2413} }{
\TR{1}
\pstree{ \TR{3142} }{
\TR{1}
\TR{1}
\pstree{ \TR{\bigominus} }{
\TR{1}
\TR{1}
}
\TR{1}
}
\pstree{ \TR{\bigoplus} }{
\TR{1}
\TR{1}
}
\TR{1}
}
&\rule{0.5in}{0pt}&
\pstree[nodesep=3pt,levelsep=25pt,linewidth=0.02in]{\TR{\bigominus} }{
\TR{1}
\pstree{ \TR{\bigoplus} }{
\pstree{ \TR{\bigominus} }{
\TR{1}
\TR{1}
}
\TR{1}
}
}
\end{array}
$$
\caption{The substitution decomposition tree of $375896214$ (of height $3$), and a pruned version of the same height which alternates between the labels $\oplus$ and $\ominus$.}\label{fig-tree-375896214}
\end{figure}

\begin{proposition}
\label{prop-bdd-subst-depth}
The substitution depth of every permutation in $\Av(\ell\dots 21)$ is at most $2\ell-3$ for $\ell\ge 2$.
\end{proposition}

\begin{proof}
We prove the result using induction on $\ell$. Only increasing permutations avoid $21$ and they have substitution depth $1$, so the result holds for $\ell=2$. Suppose to the contrary that the result holds for $\ell\ge 2$ but that there is a permutation $\pi\in\Av((\ell+1)\cdots 21)$ of substitution depth $2\ell$ (we may always assume that the depth is precisely this value, because of downward closure). If $\pi$ is a sum decomposable permutation, express it as $\pi=\pi_1\oplus\cdots\oplus\pi_m$ and note that at least one $\pi_i$ must have substitution depth $2\ell-1$.

Thus we may now assume that there is a sum indecomposable permutation $\pi\in\Av((\ell+1)\cdots 21)$ of substitution depth $2\ell-1$. Write $\pi=\sigma[\alpha_1,\dots,\alpha_m]$ where $\sigma$ is either a simple permutation or decreasing of length at least $2$ (to cover the case where $\pi$ is skew decomposable). At least one of the $\alpha_i$ must have substitution depth $2\ell-2$, and contains $\ell\cdots 21$ by induction. However, there is at least one entry of $\sigma$ which forms an inversion with $\sigma(i)$, and thus $\pi$ itself must contain $(\ell+1)\cdots 21$, as desired.
\end{proof}

It was shown in Albert, Atkinson, Bouvel, Ru\v{s}kuc and Vatter~\cite[Theorem 3.2]{albert:geometric-grid-:} that if $M$ is a forest, then $\Grid(M)$ is a \emph{geometric grid class}. In the same paper, two strong properties of geometric grid classes were established: they are defined by finitely many minimal forbidden permutations and they are strongly rational. We refer the reader to that paper for a comprehensive introduction to geometric grid classes. Our enumerative result follows from the following theorem proved in a subsequent paper.

\begin{theorem}[Albert, Ru\v{s}kuc, and Vatter~{\cite[Theorem 7.6]{albert:inflations-of-g:}}]
\label{thm-geom-inflate-enum}
The class $\C[\U]$ is strongly rational for all geometrically griddable classes $\C$ and strongly rational classes $\U$.
\end{theorem}

Applying induction on the height of substitution decomposition trees (which are bounded in all of our classes by Proposition~\ref{prop-bdd-subst-depth}), we immediately obtain the following result.

\begin{theorem}
For every $\ell$, the class $\Av(24153, 31524, \ell\cdots 21)$ is strongly rational.
\end{theorem}

\section{Concluding Remarks}

As shown in Figure~\ref{fig-wqo-results}, there are only three cases remaining: permutation graphs avoiding $\{P_6,K_5\}$, $\{P_6,K_4\}$ and $\{P_7,K_4\}$. Due to the absence of an ``obvious'' infinite antichain in these cases, we conjecture that they are all wqo. However, all three classes contain (for example) the generalised grid class $\Grid\fnmatrix{rr}{\bigoplus 21&\Av(21)}$, so these classes all contain simple permutations which are not monotone griddable. Thus our approach would have to be significantly changed to approach this conjecture.

A natural first step might be to show that the corresponding permutation classes are $\D$-griddable (in the sense of Section~\ref{sec-p7}) for a ``nice'' class $\D$. As any proof along these lines would need to use some sort of slicing argument as well (such as we used to prove Proposition~\ref{prop-gridding-corners}), $\D$ would have to be nice enough to allow for such arguments. Vatter~\cite[Lemma 5.3]{vatter:small-permutati:} has shown that such slicing arguments can be made to work when $\D$ has only finitely many simple permutations and finite substitution depth (in the sense of Section~\ref{sec-enum}). As we already have finite substitution depth by Proposition~\ref{prop-bdd-subst-depth}, it would suffice to show that we could take $\D$ to contain only finitely many simple permutations.

Again, though, this would only be an encouraging sign and not a proof, because one would then have to develop more sophisticated tools to prove wqo for such classes. Specifically, it is unlikely that the existing machinery of Brignall~\cite{brignall:grid-classes-an:} would suffice.

\def\cprime{$'$}

\end{document}